\documentclass[11pt,a4paper]{amsart}

\usepackage[utf8]{inputenc}
\usepackage{amsmath}
\usepackage{amssymb}
\usepackage{graphicx}
\usepackage{setspace}
\usepackage{centernot}
\usepackage{amsthm}
\usepackage{fullpage}
\usepackage{amsmath,amssymb,amsfonts}
\usepackage{float}
\usepackage{commath}
\usepackage{mathtools}

\usepackage{hyperref}

\numberwithin{equation}{section}

\newtheorem{theorem}{Theorem}

\newtheorem{lemma}{Lemma}
\newtheorem{conj}{Conjecture}
\newtheorem{proposition}{Proposition}

\newtheorem{defi}[]{Definition}

\renewcommand{\geq}{\geqslant}
\renewcommand{\leq}{\leqslant}

\newcommand{\om}{\omega}

\newcommand{\R}{\mathbb{R}}

\newcommand{\eps}{\varepsilon}

\newcommand{\la}{\langle}
\newcommand{\ra}{\rangle}

\newcommand{\dd}{\,\mathrm{d}}

\makeatletter
\@namedef{subjclassname@2020}{%
  \textup{2020} Mathematics Subject Classification}
\makeatother

\title{Large signed subset sums}

\author{Gergely Ambrus}
\address{
Gergely Ambrus\\
Alfr\'ed R\'enyi Institute of Mathematics, 13-15. Reáltanoda u., 1053 Budapest, Hungary}
\email[G. Ambrus]{\texttt ambrus@renyi.hu}

\author{Bernardo Gonz\'alez Merino}
\address{Bernardo Gonz\'alez Merino\\
Universidad de Murcia, Facultad de Educaci\'on, Departamento de Did\'actica de las Ciencias Matem\'aticas y Sociales, 30100-Murcia, Spain}
\email[B. Gonz\'alez Merino]{\texttt bgmerino@um.es}

\begin{document}

\begin{abstract} We study the following question: for given $d\geq 2$, $n\geq d$ and $k \leq n$, what is the largest value $c(d,n,k)$ such that from any set of $n$ unit vectors in $\R^d$, we may select $k$ vectors with corresponding signs $\pm 1$ so that their signed sum has norm at least $c(d,n,k)$?

The problem is dual to classical vector sum minimization and balancing questions, which have been studied for over a century. We give asymptotically sharp estimates for $c(d,n,k)$ in the general case. In several special cases, we provide stronger estimates:  the quantity $c(d,n,n)$ corresponds to the $\ell_p$-polarization problem, while determining $c(d, n, 2)$ is equivalent to estimating the coherence of a vector system, which is a special case of $p$-frame energies. Two new proofs are presented for the classical Welch bound when $n = d+1$. For large values of $n$, volumetric estimates are applied for obtaining fine estimates on $c(d,n,2)$.  Studying the planar case, sharp bounds on $c(2, n, k)$ are given. Finally, we determine the exact value of  $c(d,d+1,d+1)$ under some extra assumptions.
%prove the optimality of the vertex set of the regular simplex for the $c(d, d+1, d+1)$ problem under extra assumptions.
\end{abstract}

\thanks{Research of the first author was supported by NKFIH grants PD-125502 and KKP-133819.
This research is a result of the activity developed within the framework of the Programme in Support of Excellence Groups of the Regi\'on de Murcia, Spain, by Fundaci\'on S\'eneca, Science and Technology Agency of the Regi\'on de Murcia.
The second author is partially supported by Fundaci\'on S\'eneca project 19901/GERM/15 and by MICINN Project PGC2018-094215-B-I00, Spain.}



%\date{\today}

\maketitle

\section{History and results}
The study of vector sum problems dates back more than a century: see e.g. the 1913 work of Steinitz~\cite{S13} answering a question of Riemann and Lévy. In the present article, we will consider the dual of two classical problems belonging to this family.

The {\em unit vector balancing problem} asks for the following. Given unit vectors $u_1, \ldots, u_n$ in $\R^d$, one should find signs $\eps_1, \ldots, \eps_n \in \{ \pm 1 \}$ so that the sum
\[
\eps_1 u_1 + \ldots + \eps_n u_n
\]
has small norm. Equivalently, the goal is  to partition the vectors into two classes so that the corresponding partial sums are close to each other. It is natural to look for the best possible bound: in 1963, Dvoretzky~\cite{D63} asked for determining
\begin{equation}\label{maxmin_dv}
\max_{(u_i)_1^n \subset S^{d-1}} \min_{\eps \in \{ \pm 1 \}^n} \abs{\sum_{i=1}^n \eps_i u_i}
\end{equation}
for fixed $d$ and $n\geq d$, where $S^{d-1}$ denotes the unit sphere of $\R^d$. Here and later on, $|.|$ stands for the Euclidean (or $\ell_2$) norm.  Spencer~\cite{S81} gave a probabilistic proof showing that for every $n \geq d$, the above quantity equals to $\sqrt{d}$ (note that this bound is independent of the number of vectors!). Moreover, the same bound holds for sets of vectors of norm {\em at most $1$}. Sharpness is illustrated for example by taking $n = d$ and setting the vector set $(u_i)_1^d$ to be an orthonormal base of $\R^d$. Related combinatorial games were studied by Spencer~\cite{S77}. Generalizing Dvoretzky's question, Bárány and Grinberg~\cite{BG81} showed that given any set of vectors in the unit {\em ball} of a {\em $d$-dimensional normed space,} one may always find corresponding signs so that the signed sum has norm at most $d$.

Switching millennia did not halt related research: Swanepoel~\cite{S00} showed that given an {\em odd} number of {\em unit} vectors in a normed plane, there exists a corresponding signed sum of norm at most 1. Blokhuis and Chen~\cite{BC18} considered yet another twist of the problem, allowing for the coefficients to be 0 as well. Finally, we list another variation of~\eqref{maxmin_dv} due to Komlós~\cite{S87}, which has become a central question in geometric discrepancy theory.  His conjecture states that there exists a uniform constant $c$ so that for an arbitrary set $u_1, \ldots, u_n$ of (Euclidean) unit vectors in $\R^d$, one may select a sequence of signs  $\eps_1, \ldots, \eps_n \in \{ \pm 1 \}$ so that
\[
\| \eps_1 u_1 + \ldots +\eps_n u_n \|_\infty \leq c.
\]
The currently strongest upper bound is due to Banaszczyk~\cite{B98}, who proved that a sum of $O(\sqrt{\log d})$ $\ell_\infty$-norm may always be found. Further vector balancing problems are listed in the survey of Bárány~\cite{B08}.

The other question which lies at  the centre of our attention is that of {\em subset sums}. Given a set of vectors $\om_n = \{u_1, \ldots, u_n \} \subset \R^d$ of norm at most 1 {\em which sum to 0,} and a fixed $k \leq n$, the goal is to find a subset $U_k$ of $\om_n$ of {\em cardinality $k$}, so that the sum of the vectors in $U_k$ has small norm. It follows from the result of Steinitz~\cite{S13} that there always exists a subset of size $k$ whose sum has norm at most $d$. In an article of the first named author written jointly with Bárány and Grinberg~\cite{ABG16}, it is proven that for general norms, the upper bound of $\lfloor d/2 \rfloor$ holds for arbitrary $k \leq n$, whereas for the Euclidean norm, the optimal upper bound is of the order of magnitude $\Theta(\sqrt{d})$. Swanepoel~\cite{S16} considered sets of {\em unit vectors} in general Banach spaces, all of whose $k$-element subset sums have small norm.

In all the above questions, the primary task is to find sums with {\em small norm}. In the present paper, we turn our interest to the reverse direction: we set the goal to  find {\em signed sums} which are {\em large}. For doing so, we will assume that all vectors of the family in question are of norm 1, since no nontrivial bound  may hold for vectors taken from the unit ball. Unifying the two main questions above, we are going to look for {\em large signed subset sums}. When the size of the sought-after subset equals to the number of the vectors, we reach the reverse of the original unit vector balancing problem.

We only consider the Euclidean case; analogous questions for general norms may be subject to future research.

Accordingly, let us introduce the following notion.

\begin{defi}\label{def_cdnk}
  For any $d \geq 1$, $n \geq 1$ and $1 \leq k \leq n$, let $c(d,n,k)$  be the largest value so that
for every set of unit vectors $\om_n = \{u_1,\dots,u_n\}\in\mathbb S^{d-1}$ there exist indices $1\leq i_1<\dots<i_k\leq n$ and corresponding signs $\eps_1,\dots,\eps_k\in\{\pm 1\}$ such that
\begin{equation}\label{defeq_cdnk}
  \left| \eps_1 u_{i_1}+\ldots+\eps_k u_{i_k} \right| \geq c(d,n,k).
\end{equation}
\end{defi}
Equivalently, $c(d,n,k) = \min \max |\sum_{j=1}^k \eps_j u_{i_j}|$, where minimum is taken over all $n$-element vector sets $\om_n = \{ u_1, \ldots, u_n \} \subset S^{d-1}$, while maximum is taken over all $k$-element subsets of $\om_n$ and all sign sequences $\eps \in \{ \pm 1 \}^n$.
That $c(d,n,k)$ exists follows by a usual compactness argument.

Our primary interest is to estimate the quantities $c(d,n,k)$, and  determine their exact value, whenever possible.

To begin with, we note that the triangle inequality implies the trivial upper bound
\begin{equation}\label{cdnk<=k}
  c(d, n, k) \leq k,
\end{equation}
which may only be sharp if $k=1$ or $d = 1$. Another simple observation is that if $1 \leq k_1 < k_2 \leq n$, then
\begin{equation}\label{k1k2}
  c(d, n, k_1) \leq c(d, n, k_2)
\end{equation}
since given a maximal $k_1$-term sum, any additional vector may be oriented in a way so that adding it does increase the norm of the sum.
Moreover, if $2\leq d_1<d_2$, then embedding $\R^{d_1}$ into $\R^{d_2}$ implies that  $c(d_2,n,k)\leq c(d_1,n,k)$. Obviously, we also have $c(d,n_1,k)\leq c(d,n_2,k)$ for $1\leq n_1<n_2$.

%, by simply selecting
%from a maximal $n_2,k$-frame $n_1$ vectors containing $k$ vectors which attain %the maximum sum.

To obtain the simplest lower bound, we turn to an old tool. Signed sums of maximal norm appear in the argument of Bang~\cite{B51} used for the solution of Tarski's plank problem~\cite{T32}. The statement below, known as Bang's Lemma, is presented in the following form in~\cite{B01}.

\begin{proposition}[Bang]\label{bang}
If $u_1, \ldots, u_n$ are unit vectors in $\R^d$, and the signs $\eps_1, \ldots, \eps_n \in \{\pm 1\}$ are chosen so as to maximize the norm $| \sum_1^n \eps_i u_i|$, then
  \begin{equation}\label{ineq_bang}
   \la \eps_i u_i, \sum_1^n \eps_j u_j \ra  \geq 1
  \end{equation}
holds for every $i$.
\end{proposition}

Summing \eqref{ineq_bang} over $i = 1, \ldots, n$, we readily obtain that
\begin{equation}\label{lower_sqrt_n}
\Big|\sum_1^n \eps_i u_i\Big| \geq \sqrt{n}.
\end{equation}
The same estimate may alternatively be shown by taking the average over all sign sequences:
\[
\frac 1 {2^n }\sum_{\eps \in \{ \pm 1\}^n} \Big| \sum_1^n \eps_i u_i \Big|^2 =  \sum_1^n |u_i|^2  =  n,
\]
since all the mixed terms $\la \eps_i u_i, \eps_j u_j \ra$ for $i \neq j$ cancel. Moreover, the argument also shows that \eqref{lower_sqrt_n} is sharp if and only of the vector system $\om_n$ is orthonormal.

Applying \eqref{lower_sqrt_n} to any subset of $\om_n$ of cardinality $k$, we reach the first nontrivial lower bound on $c(d,n,k)$:

\begin{proposition}\label{prop2}
  For arbitrary $d, n \geq 1$ and $1 \leq k \leq n$, $c(d,n,k) \geq \sqrt{k}$ holds. This estimate is sharp if and only if $n \leq d$.
\end{proposition}

Indeed, for the estimate being sharp, we need every $k$-element subset of $\om_n$ to be orthonormal, which implies that $\om_n$ is orthonormal itself.

Among other results, we are going to give asymptotic estimates for $c(d,n,k)$. To formulate these, we will use the standard asymptotic notations $O(.), o(.), \Omega(.), \Theta(.)$ as defined in~\cite{K68}.

In the special case $k = n$, using the connection with the $p=1$ case of the $\ell_p$-polarization problem on the sphere $S^{d-1}$, the sharp asymptotic bound  was proved in~\cite{AN19}:

\begin{theorem}[A., Nietert~\cite{AN19}, Theorem 4] \label{thm_AN}
If $d, n \rightarrow \infty$ along with $n \geq d$, then
\[
c(d, n, n) = \Theta \left( \frac n {\sqrt{d}} \right).
\]
\end{theorem}
Therefore, the order of magnitude of $c(d,n,n)$ proves to be $\Theta(n)$ as opposed to the lower bound $\Omega(\sqrt{n})$ provided by Proposition~\ref{prop2}.

The next result determines the order of magnitude of $c(d,n,k)$ in the general case. %We will treat the $d = 2$ separately, see Theorem~\ref{thm_planarbound}.

\begin{theorem}\label{thm_general}
For arbitrary $d \geq 2$, $n\geq d$ and $k \geq 3$,
\begin{equation}\label{general_lower}
c(d,n,k) \geq \max \left \{  k - 8 k ^{\frac{d+1}{d-1}} n^{- \frac{2}{d-1}} , \sqrt{\frac 2 {\pi }} \cdot \frac 1 {\sqrt{d}} \cdot k  \right \}
\end{equation}
and
\begin{equation}\label{general_upper}
c(d,n,k) \leq \left\{\begin{array}{lr}
        k -  \alpha_1 \cdot \frac {1} {d^2}  \cdot k^\frac{d+1}{d-1}n^{-\frac{2}{d-1}} & \text{for } 3\leq k<  6 \cdot 100^{d-1}\\
        %k - {\alpha_1} \cdot\frac 1 {d}  k^{\frac{d+1}{d-1}} n^{- \frac{2}{d-1}} & \text{for } 3\leq k< \sqrt{n}\\
        k - \alpha_2 \cdot  k^{\frac{d+1}{d-1}} n^{- \frac{2}{d-1}} & \text{for } 6 \cdot 100^{d-1} \leq k \leq n
               \end{array}\right.
\end{equation}
if $n$ is sufficiently large, where $\alpha_1, \alpha_2>0$ are absolute constants. Furthermore, for every $k > \frac 1 {\sqrt{d}} e^{ - \frac d 2 } \cdot n$,
\begin{equation}\label{general_upper_largek}
c(d,n,k) \leq  \frac  {4\varphi}{\sqrt{\pi}} \cdot \frac{1}{\sqrt{d} }\cdot k
\end{equation}
holds with $\varphi = \sqrt{W_0\left(\frac {n^2}{k^2}\right)}$,
where $W_0$ is the principal branch of the Lambert $W$ function; equivalently, $\varphi$ is the positive solution of the equation
\begin{equation}\label{phi_kn}
  \varphi \cdot \frac k n = e^{- \frac {\varphi^2}{2}}\,.
\end{equation}
\end{theorem}

%Note that when $k = \Omega(n)$, then the second term above is the same order of %magnitude as the first, and thus the above we do not obtain a meaningful estimate.

%Therefore, we only care about the case $k = o(n)$.
%Also, when $k = O(1)$, then \eqref{eq_general} provides only an $O(1)$ estimate, which carries no information. %% I believe that when k is constant, the estimate states k-\Theta(n^{-2/(d-1)}) which is better than the basic $\sqrt{k}$

%\medskip
%**** RESULTS FOR SMALL VALUES OF $k$?? What if $k = \beta$ with some constant $\beta$??****
%\medskip

Determining the extremal vector systems $\om_n$ may only be hoped for in a few special cases. We first discuss the case $k = 2$.  Since
\begin{equation}\label{norm_u+v}
|u + v|^2 = 2 + 2\la u,v \ra
\end{equation}
for unit vectors $u, v \in S^{d-1}$, bounding $c(d, n, 2)$ is equivalent to  estimating $\max_{1\leq i <j \leq n} |\la u_i, u_j \ra|$. This is a well known problem: the quantity $\max_{1\leq i <j \leq n} |\la u_i, u_j \ra|$ is called the {\em coherence} of the vector set $\om_n$, which is a special case of the {\em $p$-frame energies} \cite{BGMPV20+, CGGKO20}. Its minimizers are called {\em Grassmannian frames} and play an important role in \emph{equiangular spherical codes} \cite{bdks18}. The classical {\em Welch bound} \cite{W74} states that one can always select $i \neq j$ so that
\begin{equation}\label{welch}
  |\la u_i, u_j \ra| \geq \sqrt{\frac{n-d}{d(n-1)}}
\end{equation}
Recently, Bukh and Cox \cite{BC20+} improved this estimate when $d+2 \leq  n < d + O(d^2)$. For the $n = d+1$ case, \eqref{norm_u+v} and \eqref{welch} imply that $c(d, d+1, 2) \geq \sqrt{2 + 2/d}$. We provide two new proofs for this estimate:

\begin{theorem}\label{thm_k=2,n=d+1}
For every $d \geq 1$,
\begin{equation}\label{c_d,d+1,2}
  c(d, d+1, 2) = \sqrt{2 + \frac 2 d}\,.
\end{equation}

The sharp bound is attained if and only if, up to sign changes, the vector system $\om_{d+1}$ forms the vertex set of a regular $d$-dimensional simplex inscribed in $S^{d-1}$.
\end{theorem}

We are also interested in the other end of the spectrum, when $n$ is much larger than $d$. In this case, the coherence bounds of Welch and Bukh-Cox are not sharp. Instead, we may turn to the {\em spherical cap packing problem}, which asks for finding the largest radius $R_{d,n}$ so that $n$ spherical caps of radius $R_{d,n}$ may be packed in $S^{d-1}$. For a detailed survey of that question, see Section 2.6 of \cite{B04}. Estimating $c(d, n, 2)$ is equivalent to bounding the packing density of non-overlapping  pairs of antipodal spherical caps of equal size. Using  volumetric estimates, we prove the following bounds for large values of  $n$:

\begin{theorem}\label{thm_k=2}
For each sufficiently large $d$, there exists an $N>0$, so that for every $n > N$,
\begin{equation}
2 - 0.51\, n^{-\frac{2}{d-1}} < c(d, n, 2) < 2 - 0.14 \, n^{-\frac{2}{d-1}}\,.
\end{equation}
\end{theorem}

Next, we study the problem in the plane. Unlike in most of the higher dimensional cases, here we are able to derive sharp bounds.

\begin{theorem}\label{thm_planarbound}
The following lower bounds hold in the plane:
\begin{equation}\label{planarbound_even}
  c(2, n, k) \geq k \cos \frac {(k-1) \pi} {2 n}
\end{equation}
for even values of $k$, and
\begin{equation}\label{planarbound_odd}
  c(2, n, k) \geq \sqrt{1 + (k-1)(k+1) \cos^2 \frac {(k-1) \pi} {2 n} }
\end{equation}
when $k$ is odd. These bounds are sharp if and only if $n$ is divisible by $(k-1)$. In those cases, equality is attained if and only if $\{ \pm u_1, \ldots, \pm u_n \}$ forms the vertex set of a regular $\frac{2n}{k-1}$-gon inscribed in $S^1$, each vertex taken with multiplicity $k-1$.
\end{theorem}

We finish our discussion with two further special cases. First, when $k =n = d$, Proposition~\ref{prop2} shows that extremizers are exactly the orthonormal systems. Second, for $ k = n = d+1$, natural intuition and numerical experiments suggest that each extremal configuration is, up to sign changes, the union of the vertex set of an even dimensional regular simplex and an orthonormal basis of the orthogonal complement of its subspace. The following conjecture,
already proven for $d=2$ in \cite[Thm. 1.3]{BFGK18}, is from \cite{BFGK18} and \cite{P19}.
Notice that $c(d,n,n)$ was considered in \cite{BFGK18} as an application to understand the behavior of the circumradius with
respect to the Minkowski addition of $n$ centrally symmetric sets in $\mathbb R^d$.

\begin{conj}\label{simplexconj}
For any $d \geq 1$, $c(d, d+1, d+1) = \sqrt{d+2}$. The sharp bound is realized if and only if, up to sign changes, $\om_{d+1}$ is the union of the vertex set of a regular simplex in a subspace $H$ centered at the origin, and an orthonormal basis of $H^\perp$, where $H$ is an even dimensional linear subspace of $\R^d$.
\end{conj}

We conclude the article with a proof of this bound under special assumptions (the problem was also posted in \cite{ol18}).

\begin{theorem}\label{thm_simplex}
Assume that $d$ is even, and the unit vectors $u_1, \ldots, u_{d+1} \in S^{d-1}$ satisfy $\sum_{i=1}^{d+1} u_i =0$. Then there exist signs $\eps_1, \ldots, \eps_{d+1} \in \{ \pm 1 \}$ so that
\begin{equation}\label{eq_simplex}
|\eps_1 u_1 + \ldots + \eps_{d+1} u_{d+1}| \geq \sqrt{d+2}.
\end{equation}
\end{theorem}

\section{General asymptotic estimates}

In this section we are going to use the notions of $\delta$-nets and $\delta$-separated sets in $S^{d-1}$. By distance we will mean  the spherical (geodesic) distance on $S^{d-1}$. Also, for $x \in S^{d-1}$ and $r \in [0,\pi]$, let $C(x,r)$ denote the spherical cap of $S^{d-1}$ with centre $x$ and radius $r$.
The volume of the unit ball $B^d$, denoted as $\kappa_d$, is given by
\[
\kappa_d = \frac{\pi^{d/2}}{\Gamma(\frac d 2 + 1)}
\]
(see e.g.  \cite{B97}). The surface area of $S^{d-1}$ is $d \,\kappa_d$. Let  $\sigma$ stand for the normalized surface area measure on $S^{d-1}$ (thus, $\sigma(S^{d-1}) = 1$).

Let $C_r$ denote a spherical cap of $S^{d-1}$ of radius $r$: that is, $C_r = C(x,r)$ for an arbitrary $x \in S^{d-1}$. By  projecting a cap radially to its boundary hyperplane and to the tangent hyperplane at its center, respectively,  we derive the simple estimates
\begin{equation}\label{sigma_C_R}
\frac {\kappa_{d-1}}{d \,\kappa_d} \sin^{d-1} r\leq \sigma(C_r) \leq \frac {\kappa_{d-1}}{d \,\kappa_d} \tan^{d-1} r.
\end{equation}
We note for later reference that Gautschi's inequality implies that
\begin{equation}\label{kappa_ineq}
 \frac{1}{\sqrt{2 \pi d}} < \frac {\kappa_{d-1}}{d \,\kappa_d} = \frac{\Gamma(\frac d 2 + 1)}{d \sqrt{\pi} \Gamma (\frac{d+1}{2})}  < \frac{1}{\sqrt{2 \pi}} \frac {\sqrt{d+2}}{d}
\end{equation}
for every $d$.

%By standard methods, we may obtain a stronger lower bound for the area of large %spherical caps: the following inequality is in close connection with the %exponential concentration property of the uniform distribution on $B^d$.
%
%\begin{lemma}\label{lem_cap}
 % Assume that $\sigma(C_R) > c_1$ with some constant $c_1 > 0$. Then there exists a % constant $c_2 >0$, independent of $d$, such that
%  \[
%  \frac {\kappa_{d-1}}{d \,\kappa_d} \cdot c_2 \sqrt{d} \cdot \sin^{d-1} R  \leq %\sigma(C_R).
%  \]
%\end{lemma}

%
%\begin{proof}
%Consider the cap $C(u,R)$, and let $T$ be the cone subtended by it, with apex at the origin. Furthermore, let $S$ be the central slab $S = \{x \in B^d: \ |\la x, u \ra| \leq \cos  R \}$. Then
%  \[
%  \sigma(C(u,R)) = \frac{\mathrm{Vol}\, T}{\kappa_d} = \frac{\kappa_d - \mathrm{Vol}\, S}{2 \kappa_d}.
%  \]
%\end{proof}

%We will also use the standard upper bound (see e.g. \cite{B97})
%\begin{equation}\label{sigma_C_R_upper2}
%\sigma(C_r) < e^{- \frac d 2   \cos^2 r}\,.
%\end{equation}

As usual, a set $X \subset S^{d-1}$ is called a {\em $\delta$-net}, if for arbitrary $y \in S^{d-1}$ there exists $x \in X$ so that the spherical distance between $x$ and $y$ is at most $\delta$. Equivalently, the spherical caps of radii $\delta$ centred at the points of $X$ completely cover $S^{d-1}$. A set $Y \subset S^{d-1}$ is {\em $\delta$-separated}, if the distance between any two of its points is at least $\delta$ -- equivalently, the spherical caps of radius $\delta/2$ centred at the points of $Y$ are pairwise non-overlapping. Furthermore, $Y$ is a {\em maximal $\delta$-separated set}, if appending any point of $S^{d-1} \setminus Y$ to it results in losing $\delta$-separatedness.  It is well known that maximal $\delta$-separated sets are (minimal) $\delta$-nets. Moreover, if $X\subset S^{d-1}$ is a maximal $\delta$-separated set, then $ |X|=\Theta(\delta^{-(d-1)}).$
This is implied by e.g. Theorem 6.3.1 of \cite{B04}, which states that if $X$ is a maximal $\delta$-separated set with $\delta < \pi /2$, then
\begin{equation}\label{card_dsep}
  \sqrt{2 \pi} \sin^{-(d-1)}\delta < |X| < 23 \, (d-1)^{3/2 } \sin^{-(d-1)} \frac \delta 2 \cdot 2^{-(d-1)/2} \,.
\end{equation}

\begin{proof}[Proof of Theorem~\ref{thm_general}]
First, we set off for the lower bound. We consider two cases depending on whether $k = \Theta(n)$ or $k = o(n)$. Note that the two terms in the estimate \eqref{general_lower} are equal when
\[
k = \left( 1 - \sqrt{\frac 2 {\pi d}} \right)^{\frac{d-1}2} \frac  n 8  \approx \frac{e^{- \sqrt{\frac{d}{2 \pi}}}}{8}  n \,;
\]
for smaller values of $k$, the first term dominates, while for larger $k$'s, the second term is larger.

Throughout the proof,  $\om_n = \{ u_1, \ldots, u_n \} \subset S^{d-1}$ will be an arbitrary $n$-element unit vector set. We will also use the notation $\pm \om_n = \{ u_1, -u_1, \ldots, u_n, -u_n \}.$

When $k$ is large, we apply the method of \cite{AN19} for proving Theorem~\ref{thm_AN}. Take $k$ vectors of $\om_n$ arbitrarily, say, $u_1, \ldots, u_k$. Then (see Proposition 3 of \cite{AN19})
\begin{equation}\label{eps_v}
  \max_{\eps \in \{ \pm 1 \}^k} \Big| \sum_{i=1}^k \eps_i u_i \Big| = \max_{v \in S^{d-1}} \sum_{i=1}^k |\la v, u_i \ra| \,.
\end{equation}
The quantity on the right hand side is easy to estimate:
\begin{align*}
\max_{v \in S^{d-1}} \sum_{i=1}^k |\la v, u_i \ra| & \geq \int_{S^{d-1}}  \sum_{i=1}^k |\la v, u_i \ra| \dd \sigma(v) \\
& = k \int_{S^{d-1}} |\la v, u_1 \ra| \dd \sigma(v)\\
&= k \frac{2}{d \kappa_d} \int_0^1 t (1 - t^2)^{\frac{d-3}{2}} (d-1) \kappa_{d-1} \dd t\\
& = k \frac{2 \kappa_{d-1}}{d \kappa_d} \\
& \geq k \sqrt{\frac{2}{\pi d}},
\end{align*}
which is the second term of the estimate in \eqref{general_lower}.

Next, we establish the estimate for small $k$'s. Let $r$ so that
  \[
  \sigma(C_r) = \frac {k}{2n}.
  \]
By \eqref{sigma_C_R} and \eqref{kappa_ineq},
\[
\frac {1}{\sqrt{2 \pi d}} \sin^{d-1} r < \frac k {2n}\,.
\]
Using that on $[0, \pi/2]$, $ \frac 2 \pi x \leq  \sin x$,  this implies that
\[
r < 4 \Big( \frac k n \Big)^{\frac{1}{d-1}}\,.
\]
Accordingly, for any $x \in S^{d-1}$ and for every $y \in C(x, r)$, we have
\begin{equation}\label{xyest}
  \la x, y  \ra \geq \cos r > 1 - 8 \Big( \frac k n \Big)^{\frac 2 {d-1}}.
\end{equation}

Let us denote by $\#(X)$ the cardinality of a finite set $X\subset\mathbb R^d$. Since
\[
 \int_{S^{d-1}} \#( (\pm \om_n) \cap C(x,r)) \dd \sigma(x) = \sum_{u \in \pm \om_n} \sigma(\{ x \in S^{d-1}: \ u \in C(x, r) \}) = 2n \cdot \frac{k}{2n} = k,
\]
there exists some $x \in S^{d-1}$ for which at least $k$ vectors of $\pm \om_n$ lie in $C(x,r)$. Let $v_1, \ldots, v_k \in \pm \om_n$ be $k$ such vectors. Then, by \eqref{xyest},
\begin{equation}\label{sumv}
  |v_1 + \ldots + v_k| \geq \sum_{i=1}^k \la x, v_i \ra \geq k  - 8 k^{\frac{d+1}{d-1}} n^{- \frac{2}{d-1}}\,,
\end{equation}
which is the first term of desired lower bound.

Now we turn to the upper bounds. First, we show \eqref{general_upper}. Take $\om_n$ to be a $\delta$-separated set of $n$ points in $S^{d-1}$ with $\delta$ being as large as possible. By \eqref{card_dsep}, \begin{equation}\label{deltaest}
  n^{-\frac{1}{d-1}}<\delta <  \, 33 n^{-\frac{1}{d-1}}
\end{equation}
when $n$ is sufficiently large. Note that the spherical caps of radius $\delta/2$ centred at the points of $\om_n$  are pairwise non-overlapping.

Let now $u \in S^{d-1}$ be the unit direction vector of the largest $k$-term signed subset sum of $u_1, \ldots, u_n$, that is, the largest $k$-term sum of $\pm \om_n$. Then the norm of this maximal sum equals to the sum of the $k$ largest inner products of the vectors of $\pm \om_n$ taken with~$u$.

Our goal is to find a radius $R$ so that the cap $C(u, R )$ may contain at most $\gamma k$ points of $\pm \om_n$ with a parameter $\gamma \in (0,1)$. This  guarantees that
\begin{equation}\label{cdnk_gammak}
  c(d,n,k) \leq \gamma k + (1- \gamma) k \cos R.
\end{equation}
Note that the open spherical caps of radius $\delta/2$ centred at the points of $\pm \om_n \cap C(u, R)$ are all contained in $C(u, R + \frac{\delta}2)$. On the other hand, any point of $C(u, R + \frac{\delta}2)$ may be covered by at most two of the interiors of these caps. Thus, we deduce $\#(C(u,R) \cap \pm \om_n) \leq \gamma k$, and therefore that \eqref{cdnk_gammak} holds, if $R$ and $\gamma$ satisfy
\begin{equation}\label{Rgamma}
  \sigma\left(C_{ R + \delta / 2}\right) < \frac \gamma 2 k  \cdot \sigma( C_{\delta/2}).
\end{equation}
This is what we will show.

We are going to divide the argument again into two parts according to the magnitude of $k$, as different parameters will be needed depending on the range.

Let us first assume that $ 3 \leq k < 6 \cdot 100^{d-1}$. Define $R_1$ to be
\begin{equation}\label{R_1}
  R_1 = \frac{\delta}{2}\left(\left(\frac{3k}{8}\right)^{\frac{1}{d-1}}-1\right).
\end{equation}
Notice that $R_1>0$ due to $k\geq 3$.
By \eqref{deltaest} then
\begin{equation}\label{Rasymp}
  R_1 = \Theta\Big( \Big(\frac k n  \Big)^{\frac 1 {d-1}} \Big).
\end{equation}
In particular, $R_1 \rightarrow 0$ as $n \rightarrow \infty$. Since $\delta \rightarrow 0$ as well, for any $\alpha>0$ we have that
\begin{equation}\label{R_delta}
\tan \Big(R_1 + \frac \delta 2 \Big) \leq (1+\alpha) \Big(R_1 + \frac \delta 2 \Big)
\end{equation}
and
\[
\sin (\delta/2) \geq \frac{\delta}{2(1+\alpha)}
\]
for large enough $n$. Thus, by \eqref{sigma_C_R},
\[
\begin{split}
 \frac{\sigma(C_{ R_1 + \delta/2})}{\sigma(C_{\delta/2})}
\leq  \frac {\tan^{d-1} (R_1 + \delta/2)}{\sin^{d-1}(\delta/2)}
 & \leq  (1+\alpha)^{2(d-1)} \Big( \frac{2R_1}{\delta} + 1 \Big)^{d-1} \\
 & =  (1+\alpha)^{2(d-1)}\cdot \frac{3k}{8}\\
 &= k \cdot \frac 7 {16},
\end{split}
\]
with $\alpha= (\frac 7 6 )^{1/(2(d-1))} -1$, for sufficiently large $n$. This shows that \eqref{Rgamma} holds with $R = R_1$ and $\gamma = \frac 7 8$.
Since $R_1=\frac\delta2k^\frac{1}{d-1}\left(\left(\frac38\right)^{\frac{1}{d-1}}-\left(\frac1k\right)^\frac{1}{d-1}\right)$,
we obtain that
\[
\begin{split}
R_1 \geq \frac\delta2k^\frac1{d-1}\left(\frac38-\frac1k\right)\frac{1}{d-1}\left(\frac38\right)^{\frac{2-d}{d-1}}
    \geq \frac\delta {48 \, d} \cdot k^\frac1{d-1}
\end{split}
\]
for $k\in[3,\sqrt{n}]$.
Using that $\cos x < 1 - x^2 /4$ for $x \in [0, \frac \pi 2]$, by \eqref{cdnk_gammak} we deduce that
\[
c(d,n,k) \leq \frac{7k}{8} + \frac{k}{8} \left(1-\frac{R_1^2}{4}\right) < k - \frac{1}{8\cdot 48^2 d^2}k^\frac{d+1}{d-1}n^{-\frac{2}{d-1}}
%< k - \frac 1 {24 d} \left( \frac 3 8 \right)^{\frac 1 {d-1}} k^{\frac{d+1}{d-1}} n^{- \frac{2}{d-1}} \leq  k - \frac 1 {64 d}  k^{\frac{d+1}{d-1}} n^{- \frac{2}{d-1}}
\]
if $n$ is large enough. This establishes the first estimate of \eqref{general_upper} with $\alpha_1 = \frac{1}{8 \cdot 48^2}$.

 Next, we assume that $6 \cdot 100^{d-1} \leq k \leq n $.
Define $R_2$ so that
\begin{equation}\label{R2eq}
 \tan\left(R_2 + \frac \delta 2 \right)  = \frac 1 {2} \left( \frac k { 6n}  \right)^{\frac 1 {d-1}}.
\end{equation}
Note that for sufficiently large $n$, $R_2 > \delta$, since by  \eqref{deltaest},
\[
\tan \frac {3 \delta} 2 <  \frac {100}{99} \cdot \frac{3\delta}{2} < \frac{100}{2} n^{\frac{-1}{d-1}} \leq \frac{1}{2}\left(\frac k { 6n} \right)^{1/(d-1)}.
\]
Let $\alpha = 2^{\frac{1}{d-1}} - 1$. Then as above, for sufficiently large values of $n$, we obtain using \eqref{sigma_C_R} and \eqref{deltaest} that
\begin{align*}
\frac{\sigma(C_{ R_2 + \delta/2})}{\sigma(C_{\delta/2})}\leq  \frac {\tan^{d-1} (R_2 + \delta/2)}{\sin^{d-1}(\delta/2)}
 & \leq  2^{-(d-1)} \frac k { 6n} \cdot (1+\alpha)^{d-1} \left( \frac \delta 2 \right)^{-(d-1)} < \frac{ k}{3}\,.
\end{align*}
Therefore, \eqref{Rgamma} holds with $R = R_2$ and $\gamma= \frac 2 3$. Note that \eqref{R2eq} shows that $\tan(R_2 + \frac{\delta}{2}) \leq \frac 1 {12}$, and hence
\[
R_2 > \frac{2}{3}\left(R_2 + \frac{\delta}{2}\right) > \frac{2}{3} \cdot \frac 3 4 \tan\left(R_2 + \frac{\delta}{2}\right) = \frac 1 {4} \left( \frac k { 6n}  \right)^{\frac 1 {d-1}}
%R_2 + \frac{\delta}{2} < \frac {10}{9} \tan\left(R_2 + \frac{\delta}{2}\right) = \frac 1 {9} \left( \frac k n  \right)^{\frac 1 {d-1}}
\]
for large enough $n$. Thus, \eqref{cdnk_gammak} implies (using that for small enough $x$, $\cos x < 1 - \frac {x^2} 4$ holds) that
\[
c(d,n,k) \leq \frac 2 3 k + \frac 1 3  k \cos R_2 \leq k -  k \cdot \frac 1 {64} \left( \frac k { 6n}  \right)^{\frac 2 {d-1}} \leq k - \frac{1}{64 \cdot 36}  k^{\frac{d+1}{d-1}} n^{- \frac{2}{d-1}} \,,
\]
which is the second estimate of \eqref{general_upper} with $\alpha_2 = \frac 1 {64 \cdot 36}$.

Finally, we establish \eqref{general_upper_largek}. Accordingly, assume that $k > \frac{1}{\sqrt{d}}e^{-\frac{d}{2}}\cdot n$.
Take $\om_n$ as before. By Theorem 6.1.6 of~\cite{BHS19}, $\om_n$ is  uniformly distributed,  i.e. for every closed set  $D\subset{S^{d-1}}$ with zero-measure relative boundary, \[
 \lim_{n \rightarrow \infty} \frac{\# (\om_n \cap D)}{n} = \sigma(D)
 \]
holds. Thus, $\pm \om_n$ is uniformly distributed on $S^{d-1}$ as well.

Let now $\eps\in (0, 1)$ be fixed, whose values we will set later.
A standard compactness argument yields that for large enough $n$
\begin{equation}\label{capest}
\frac{1}{1 + \eps} \cdot 2n \cdot \sigma(C)  < \#( \pm \om_n \cap C) < (1 + \eps) \cdot 2 n \cdot \sigma(C)
\end{equation}
is valid for every spherical cap $C$ with $\sigma(C) > \frac 1 {4 \sqrt{d}}e^{- \frac d 2 } $. We will assume this property from now onwards.

Let again $u \in S^{d-1}$ be the unit direction vector of the largest $k$-term sum of $\pm \om_n$. If the elements of $\pm \om_n$ are ordered according to their inner products with $u$ in decreasing order, then $
|u| = \sum_{i=1}^k \la u, v_i \ra.
$
Thus, if $\rho$ denotes the spherical distance between $u$ and $v_k$, then
\begin{equation}\label{unorm}
  |u| \leq \sum_{v \in ( \pm \om_n \cap C(u, \rho) ) } \la u, v \ra.
\end{equation}
Here, the interior of $C(u, \rho)$ contains strictly less than $k$ points of $\pm \om_n$, while $C(u, \rho)$ contains at least $k$ points of $\pm \om_n$. Because of \eqref{capest}, this shows that
\begin{equation}\label{urho}
\frac 1 {4 \sqrt{d}} \cdot e^{- \frac d 2}< \frac{1}{1 + \eps} \cdot \frac{k}{2 n} \leq \sigma(C(u, \rho)) \leq (1 + \eps) \frac k {2 n} \,.
\end{equation}
The symmetry of $\pm \om_n$ also implies that $\sigma(C(u, \rho)) \leq \frac 1 2$.

Now, Theorem 13.3.1 of~\cite{BHS19} shows that the discrete probability measure with equal point masses at the elements of $\pm \om_n$ converges to $\sigma(.)$ in the weak*-topology. This implies (see e.g. Theorem 1.6.5. in the same reference) that for a fixed spherical cap $C \subset S^{d-1}$,
\[
\frac {1}{2n} \sum_{v \in ( \pm \om_n \cap C ) } \la u, v \ra \rightarrow \int_{C} \la u, w \ra \dd \sigma(w)
\]
as $n \rightarrow \infty$. Therefore, there exists an index $N$ so that for every $n \geq N$ and for every spherical cap $C \subset S^{d-1}$ with
%$\sigma(C) > \frac {\beta} 4$
%$\sigma(C) > \frac {\beta} {2.002}$,
$ \sigma(C) >   \frac 1 {4 \sqrt{d}} \cdot e^{- \frac d 2}$,
\begin{equation}\label{vsumint}
  \frac{1}{2n} \sum_{v \in ( \pm \om_n \cap C ) } \la u, v \ra \leq (1 + \eps) \int_{C} \la u, w \ra \dd \sigma(w).
\end{equation}
Let now $R_3$ be the radius so that
\begin{equation}\label{eq_R3}
  \sigma(C_{R_3}) = ( 1 + \eps) \frac k {2 n} .
\end{equation}
Let $\varphi$ be defined by \eqref{phi_kn}. Since the function $f(x) = \frac 1 x e^{-\frac {x^2}2}$ is monotonically decreasing on $[0, \infty)$, the condition $k > \frac 1 {\sqrt{d}} e^{ - \frac d 2 } \cdot n $ ensures that $\varphi < \sqrt{d}$. On the other hand, since $k \leq n$, we also have that $\varphi \geq \sqrt{W_0(1)} \approx 0.7531$.
%Notice that $\left(1-1/(c\beta d)\right)^{\frac{d-1}{2}}\leq\sqrt{\beta}$ holds true for every $\beta\in[1/(cd),1]$ and some $c\geq 1$, if and only if $f(\beta)=\beta(1-\beta^\frac{1}{d-1})\leq 1/(cd)$. Since $f(\beta)\leq f((1-1/d)^{d-1})=(1-1/d)^{d-1}/d$,
%thus the inequality holds if and only if $(1-1/d)^{d-1}/d \leq 1/(cd)$, i.e $c \leq 1/(1-1/d)^{d-1}$ for every $d\geq 2$ and thus
%$c\leq 2$. Therefore
%\begin{equation}\label{eq_c_2}
%    \left(1-\frac{1}{2\beta d}\right)^{\frac{d-1}{2}}\leq\sqrt{\beta} \quad \text{ for every }\beta\in[\frac{1}{2d},1]\text{ and }d\geq 2.
%\end{equation}

Then, by  \eqref{phi_kn},  \eqref{kappa_ineq}, \eqref{unorm}, \eqref{urho}, \eqref{vsumint}, %\eqref{eq_c_2},
and \eqref{eq_R3}, and using the standard estimate
\[
\left( 1- \frac{\varphi^2}{d} \right)^{\frac{d-1}{2}} < e^{-\frac{\varphi^2}{2}}
\]
which holds for every $d \geq 2$ and $0 < \varphi < \sqrt{d}$, we have that
\begin{align*}
|u| &\leq \sum_{v \in ( \pm \om_n \cap C(u, \rho) ) } \la u, v \ra \leq 2(1 + \eps) n \int_{C(u, \rho)} \la u, w \ra \dd \sigma(w) \\
& \leq 2 (1 + \eps)  n \int_{C(u, R_3)} \la u, w \ra \dd \sigma(w) \\
&=  2 (1 + \eps)  n \cdot\frac{(d-1) \kappa_{d-1}}{d \, \kappa_d} \int_{\cos R_3}^1 t (1 - t^2)^{\frac{d-3}{2}} \dd t\\
&\leq  2(1 + \eps)  n \cdot \frac { (d-1) \, \kappa_{d-1}}{d \, \kappa_d} \left(
\frac {\varphi}{\sqrt{d}}  \int_{\cos R_3}^1  (1 - t^2)^{\frac{d-3}{2}}  \dd t   +
\int_{ \frac {\varphi}{\sqrt{d}} }^1 t (1 - t^2)^{\frac{d-3}{2}} \dd t
\right)
\\
&=  2(1 + \eps)  n \cdot \frac {\varphi}{\sqrt{d} } \cdot \sigma(C_{R_3}) + 2(1 + \eps) n \cdot \frac { \kappa_{d-1}}{d \, \kappa_d} \left( 1 - \frac {\varphi^2} { d} \right)^{\frac{d-1}{2}}\\
& \leq (1 + \eps) ^2 \cdot \frac{\varphi}{\sqrt{d} }  \cdot k + 2(1 + \eps)  n \cdot \frac 1 {\sqrt{2 \pi}} \cdot \frac {\sqrt{d+2}}{d} \cdot e^{ - \frac{\varphi^2}{2}}\\
&=  \frac{\varphi}{\sqrt{d} }  \cdot k \left( (1 + \eps) ^2 + 2(1 + \eps) \frac 1 {\sqrt{2 \pi}} \cdot \sqrt{\frac {d+2}{d}} \right) \\
&\leq\frac  {4\varphi}{\sqrt{\pi}} \cdot \frac{1}{\sqrt{d} }\cdot k,
\end{align*}
where we set $\eps$ to be the positive solution of the quadratic equation
\[
(1+ \eps)^2 + \frac 2 {\sqrt{\pi}} (1 + \eps)  = \frac {4}{\sqrt{\pi}}\,.
\qedhere
\]
\end{proof}

We note that an estimate  of the same order of magnitude than \eqref{general_upper_largek} for the reduced range $k \geq \frac 1 d \cdot n$ may be obtained as follows. Assume that $n = m d$, and take $\om_n$ to be  $m$ copies of an orthonormal base in $\R^d$. It is not hard to see that if $a \geq 1 $ is an integer, then the sum of any $a \cdot m$ signed vectors of $\om_n$ has norm at most $\sqrt{a} \cdot m$. Thus, if we set $k = a \cdot m$, then we readily see that any $k$-term signed sum of $\om_n$ has norm at most
\[
\sqrt{a} \cdot m = \sqrt{ \frac {d}{a}} \cdot \frac{1}{\sqrt{d}}\cdot k
\]
which is slightly stronger than \eqref{general_upper_largek} if $a> d \cdot \frac{\pi}{4 \varphi^2}$. The estimate may be then extended to every $k \geq \frac 1 d \cdot n$ using the monotonicity property \eqref{k1k2}.

\section{Selecting two vectors}
We start this section by presenting two new, essentially different proofs for the estimate of  $c(d, d+1,2)$ yielded by the Welch bound. The first uses linear dependences (for several beautiful applications of that method, see~\cite{B08}).

\begin{proof}[First proof of Theorem~\ref{thm_k=2,n=d+1}]
By means of \eqref{norm_u+v}, it suffices to show that for any set of $d+1$ unit vectors $u_1, \ldots, u_{d+1} \in S^{n-1}$, there exist  indices $i \neq j \in [d+1]$ so that
  \[
  |\la u_i, u_j \ra| \geq \frac 1 d \,.
  \]
Since the number of vectors exceeds $d$, they must be linearly dependent: there exist reals $c_1, \ldots, c_{n+1}$, not all $0$, so that
\begin{equation*}\label{cui}
  \sum_{i=1}^{d+1} c_i u_i = 0.
\end{equation*}
Taking norm squares and using that $|u_i|^2 =1$ leads to
\begin{equation*}\label{cui2}
  \sum_{i=1}^{d+1} c_i^2 = - 2 \sum_{i<j} c_i c_j \la u_i, u_j \ra.
\end{equation*}
Let $M = \max_{i<j} |\la u_i , u_j \ra|$. Then
\begin{equation*}\label{cui3}
  \sum_{i=1}^{d+1} c_i^2 \leq 2 \sum_{i<j} |c_i| |c_j| M,
\end{equation*}
thus
\begin{equation}\label{Mest}
  M \geq \frac {\sum_{i=1}^{d+1} c_i^2 }{ 2 \sum_{i<j} |c_i| |c_j|}\,.
\end{equation}
By rescaling, we may assume that $\sum_{i=1}^{d+1} c_i^2 = 1$. Then,
\begin{equation}\label{cicj}
  2 \sum_{i<j} |c_i| |c_j| = \left( \sum_{i=1}^{d+1} |c_i| \right)^2 - \sum_{i=1}^{d+1} c_i^2 \leq (d+1) - 1 =d
\end{equation}
by the inequality between the arithmetic and quadratic means. Therefore, \eqref{Mest} shows that $M \geq \frac 1 d$, which was our goal.

When extremum holds, all the above inequalities must be equalities. This means that $M = \frac 1 d$. Therefore, by \eqref{cicj}, $|c_i| = \sqrt{d+1}$ for every $i$. Also, by \eqref{cui3}, $|\la u_i, u_j \ra| = M$ for every $i \neq j$. Thus, up to sign changes, the vectors $(u_i)_1^{d+1}$ must form the vertex set of a regular simplex centered at the origin. In this case, the larger signed sum of any two of the vectors indeed has norm $\sqrt{{2 (d+1)} / {d}}$.
\end{proof}

The second, independent proof uses induction on $d$, and applies Jung's theorem.

\begin{proof}[Second proof of Theorem~\ref{thm_k=2,n=d+1}]
The result clearly holds for $d=1$. We are going to prove it for $d$, assuming its validity for dimensions up to $d-1$.

Consider now an arbitrary set $\om_{d+1} = \{ u_1,\dots u_{d+1}\}$. Let $L = \mathrm{lin} \, \om_{d+1}$, its linear span. If $\mathrm{dim}(L)=m\leq d-1$, then we may assume w.l.o.g. that $L = \mathrm{lin} \{u_1, \ldots, u_{m+1}  \} $. The inductive hypothesis applied for the vector set $\{ u_1, \ldots, u_{m+1} \}$ implies that there exist signs $\eps_1, \eps_2$ and indices $1\leq i_1<i_2\leq m+1$, so that
\[
| \eps_1 u_{i_1}+\eps_2 u_{i_2}| \geq\sqrt{\frac{2(m+1)}{m}}>\sqrt{\frac{2(d+1)}{d}},
\]
proving the assertion.

Hence, we may suppose that $\mathrm{dim}(L)=d$ with $u_2,\dots,u_{d+1}$ being linearly independent.
Let $\rho>0$ be such that
\[
-\rho u_1\in\partial \mathrm{conv}(\{\pm u_2,\dots,\pm u_{d+1}\})=\partial P,
\]
where $P=\mathrm{conv}(\{\pm u_2,\dots,\pm u_{d+1}\})$, and where
$\mathrm{conv}$ and $\partial$ stand for convex hull and boundary, respectively.
Since $P$ is a polytope with $0\in\mathrm{int}(P)$, $-\rho u_1$ belongs to a facet of $P$, which does not contain $0$.
Hence there exist signs $\eps_2,\dots,\eps_{d+1}$ such that
\[
-\rho u_1\in\mathrm{conv}(\{\eps_2u_2,\dots,\eps_{d+1}u_{d+1}\}).
\]
In particular we have that
\[
0\in\mathrm{conv}(\{u_1,\eps_2u_2,\dots,\eps_{d+1}u_{d+1}\})=S,
\]
thus meaning that $\mathrm{R}(S)=1$, where $\mathrm{R}(.)$ stands for the circumradius (see \cite[Proposition 2.1]{BG17}). Hence, if $\mathrm{D}(S)$ denotes the diameter of $S$, then by Jung's theorem \cite{Ju01}
(see also \cite[Lem. 3]{Ba92} or \cite[(3)]{BG2}) we have that
\[
\frac{\mathrm{D}(S)}{\mathrm{R}(S)}\geq\sqrt{\frac{2(d+1)}{d}}\,.
\]
Since the diameter of $S$ is attained between two vertices of $S$, this means that
either
\[
\left| u_1-\eps_iu_i \right| \geq\sqrt{\frac{2(d+1)}{d}}\quad\text{or}\quad \left| \eps_{i_1}u_{i_1}-\eps_{i_2}u_{i_2} \right| \geq\sqrt{\frac{2(d+1)}{d}},
\]
for some $i\in\{2,\dots,d+1\}$ or some $2\leq i_1<i_2\leq d+1$, yielding the assertion.

If equality holds, then we must have equality in Jung's theorem. Therefore, the set of vertices
$\{u_1,\eps_2u_2,\dots,\eps_{d+1}u_{d+1}\}$ form the vertex set of a regular simplex, as desired.
\end{proof}

Next, we turn to estimates for large values of $n$.

\begin{proof}[Proof of Theorem~\ref{thm_k=2}]
We will use the fact that for two unit vectors $u,v \in S^{d-1}$ of geodesic distance $\delta \leq \frac \pi 2$, we have $|u + v| = 2 \cos \frac \delta 2$, and therefore
\begin{equation}\label{u+v}
   2 - \frac {\delta^2}{4} <|u + v|  < 2 - \frac{\delta^2}{5} .
\end{equation}

First, we show the lower bound. To this end, by \eqref{u+v}, it suffices to show that if $n$ is large enough, then  for any set $\om_n = \{ u_1, \ldots, u_n \} \subset S^{d-1}$, there exist two vectors of the set $\pm \om_n$ of $2n$ vectors, whose geodesic distance is at most $\delta_1:=1.001\cdot \sqrt{2} n^{-1/(d-1)}$. This is indeed guaranteed by \eqref{card_dsep}, since if $d$ is large enough, the maximal cardinality of a $\delta_1$-separated set in $S^{d-1}$ may not exceed
\[
23 \, (d-1)^{3/2 } \sin^{-(d-1)} \frac {\delta_1} 2 \cdot 2^{-(d-1)/2}  < 1.001^{d-1}
\Big( \frac {\sqrt{2}}{\delta_1}\Big)^{d-1} =n.
\]
Now, let us turn to the upper bound. Set $\delta_2 = 0.75 \cdot n^{-1/(d-1)}$, and let $X = \{x_1, \ldots, x_m \}$ be a maximal $ \delta_2$-separated subset of $S^{d-1}$ with respect to the the Euclidean distance this time. Then $X$ is a maximal $\delta_2'$-separated subset of $S^{d-1}$, with $\delta_2 < \delta_2' < 1.001 \, \delta_2$ if $n$ is large enough. Consequently, by \eqref{card_dsep},
\begin{equation}\label{m_est}
  m > \sqrt{2 \pi}  \sin^{-(d-1)} (1.001 \,\delta_2)> (1.001\,\delta_2)^{-(d-1)}>1.332^{d-1} n.
\end{equation}

Define a graph $G$ on the vertex set $[m] = \{1, 2, \ldots, m\}$ as follows: two non-equal indices $i$ and $j$ are connected if and only if  $|x_i + x_j| < \delta_2$. We are going to bound the maximal degree in $G$.

For every $i \in [m]$, let $B_i$ be the $d$-dimensional ball of radius $\delta_2 / 2$ centred at $x_i$. Take now an arbitrary $i \in [m]$, and let $B'$ be the ball of radius $\delta_2/2$ centred at $-x_i$. Note that if $ij$ is an edge of $G$, then $B'$ and $B_j$ intersect. On the other hand, $B_j$ and $B_{j'}$  do not overlap when $j \neq j'$. It readily follows by a simple geometric argument that the number of $j$'s connected to $i$ is at most $\tau_d$ +1, where $\tau_d$ is the {\em kissing number} in $d$ dimensions: the maximal number of non-overlapping equal-sized spheres in $\R^d$, all touching a central sphere of the same size. Indeed, by using the fact that $X$ is $ \delta_2$-separated, there exists at most one index $j$ so that $|x_i + x_j| < \delta_2 / 2$. On the other hand, assuming that both $j$ and $j'$ are adjacent to $i$, the distance between $-x_i$ and $x_j$, resp., $x_{j'}$ is at least $\delta_2 /2$, and using that $|x_{j} - x_{j'}| \geq \delta_2$, we see that projecting $x_j$ and $x_{j'}$ radially from $x_i$ to the sphere of radius $\delta_2$ centred at $x_i$ does not decrease their distance.

A classical result of Kabatiansky and Levenshtein \cite{KL78} states that
\[
\tau(d) \leq 2^{0.401d(1+o(1))}.
\]
Accordingly, for large enough $d$, we may assume that $\tau_d < 1.33^d -2$. Then, by the previous arguments, the maximum degree of $G$, denoted by $\Delta$, is at most $\tau_d + 1 < 1.33^d-1$.

Let $Y$ be a maximal independent set in $G$. The cardinality of $Y$ is the {\em independence number} of~$G$, which is at least $m/(1 + \Delta)$ by Brook's bound \cite{B41,B60}. Accordingly, using \eqref{m_est}, we obtain that
\[
|Y| \geq \frac m {1 + \Delta} \geq \frac {1.332^{d-1} n}{1.33^d} >n.
\]
Thus, we may set $\om_n$ to be a set of  $n$ distinct vectors from $Y$. For any $x_i, x_j \in \om_n$, we have that $|x_i - x_j| \geq \delta_2$ and $|x_i + x_j| \geq \delta_2$. Then, using that $|x_i - x_j|^2 + |x_i + x_j|^2 =4$,
\[
\max\{|x_i - x_j|^2,|x_i + x_j|^2  \} \leq 4 - \delta_2^2,
\]
which implies that any signed sum of two distinct elements of $\om_n$ is bounded above by
\[
2 - \frac{\delta_2^2}{4} < 2 - 0.14\, n^{- \frac{2}{d-1}} \,.
\qedhere
\]

We note that a slightly weaker upper bound may be obtained simply by defining edges of $G$ corresponding to pairs $i,j$ such that $|x_i + x_j| < \delta_2 /2$. Then
the degree of any vertex may be at most 1, which would result in an $m/2$ lower bound for the independence number of $G$. This advantage is, however, balanced out by the weaker distance bound.
\end{proof}

\section{Sharp bounds in the plane}
First, we state and prove a simple lemma. To that end, we identify $S^1$ with the complex unit circle. For any $u \in S^1$ of the form $u = e^{i \psi}$, $\psi$ is called the {\em angle} of $u$.

\begin{lemma}\label{lemma_arc}
Assume that $\varphi \in [0, \pi]$, and that the unit vectors $u_1, \ldots, u_k \in S^1$ all have angles in the interval $[0, \varphi]$. Then
\begin{equation}\label{minnorm_even}
  \left| u_1 + \ldots + u_k \right| \geq k \cos  \frac \varphi 2
\end{equation}
when $k$ is even, and
\begin{equation}\label{minnorm_odd}
  \left| u_1 + \ldots + u_k \right| \geq \sqrt{1 + (k-1)(k+1) \cos^2 \frac {\varphi} {2}}
\end{equation}
when $k$ is odd.
\end{lemma}
\begin{proof}
Let $v = e^{i \varphi/2}\in S^1$. Then
\[
\la u_i, v \ra \geq \cos  \frac \varphi 2
\]
holds for every $i \in [k]$. Consequently,
\[
\Big \la \sum_{i=1}^k u_i, v  \Big \ra \geq k \cos  \frac \varphi  2 \,.
\]
Clearly, this quantity is also a lower bound on the norm of $ \sum_{i=1}^k u_i$. When $k$ is even,  we reach \eqref{minnorm_even}. That this bound is sharp is shown by considering the family consisting of $k/2$ copies of $1$ and $k/2$ copies of $e^{i \varphi}$.

Let us now assume that $k$ is odd, and that $u_1, \ldots, u_k$ are so that $ \sum_{i=1}^k u_i = u = e^{i \psi}$ has minimal norm.
 Clearly, $\psi \in [0, \varphi]$. We may assume that the vectors $u_1, \ldots, u_l$ are of angle at most $\psi$, while $u_{l+1}, \ldots, u_k$ have angle in $(\psi, \varphi]$. It is easy to see that replacing any $u_i$ with $1$ for $i\leq l$, or any $u_j$ with $e^{i \varphi}$ for $j \in [l+1, k]$ results in decreasing the norm of $\sum_{i=1}^k u_i $, unless the vectors to be replaced are already $1$ or $e^{i \varphi}$. Therefore, the extremal systems consist of copies of these two vectors at the ends of the circular arc of length~$\varphi$. If $l < (k-1)/2$, then swapping one copy of $e^{i \varphi}$ by $1$ again results in decreasing the norm of the sum. Thus, by symmetry, we conclude that the minimum norm is attained at the vector system consisting of $(k-1)/2$ copies of $1$ and $(k+1)/2$ copies of~$e^{i \varphi}$. Applying the law of cosines leads to \eqref{minnorm_odd}, which is again a sharp bound.
\end{proof}

\begin{proof}[Proof of Theorem~\ref{thm_planarbound}]
In light of Lemma~\ref{lemma_arc}, it suffices to show that from any set of $2n$ unit vectors of the form $\{u_1, -u_1,  \ldots, u_n, -u_n \} \subset S^1$ , we may select $k$ vectors which belong to an arc of angle not larger than $(k-1) \pi / n$. Order the vector system with respect to positive orientation along the circle, and re-index the vectors as $v_1, v_2, \ldots, v_{2n}$ according to this ordering. For any $i \in [1, 2n]$, let $\alpha_i$ be the angle between $v_i$ and $v_{i+1}$, and $\beta_i$ be the angle between $v_i$ and $v_{i+k-1}$ (with the indices being understood modulo $2 n$). Then $\sum_{i=1}^{2n} \alpha_i = 2 \pi$, and for every $i$,
\[
\beta_i = \sum_{j=0}^{k-2} \alpha_{i + j}.
\]
Hence,
\[
\sum_{i=1}^{2n} \beta_i = (k-1) \sum_{i=1}^{2n} \alpha_i  = 2 (k-1) \pi.
\]
Thus, there exists an index $i$ for which $\beta_i \leq (k-1)\pi/n$, which was our goal to prove.

Now, let us turn to the case of equality. The above bound may only be sharp for vector systems for which $\beta_i = (k-1)\pi/n$ for every $i$. On the other hand, in order for the estimate of Lemma~\ref{lemma_arc} being sharp, one needs that $\{ \pm u_1, \ldots, \pm u_n\}$ consists of copies of a given point set with multiplicity $k/2$ (for even values of $k$), or with multiplicity at least $(k-1)/2$ and $(k+1)/2$, alternatingly (for odd values of $k$). Combining these two conditions, we deduce that the vector system $\{ \pm u_1, \ldots, \pm u_n \}$ must be the vertex set of a regular $ ((2n)/(k-1))$-gon inscribed in $S^1$, with each of its vertices being taken with multiplicity $k-1$. Given that $\{ \pm u_1, \ldots, \pm u_n \}$ is an antipodal set, we also derive that $(2n)/(k-1)$ must be even, that is, $k-1$ must divide $n$. When this condition holds, the vector system described above indeed yields equality in \eqref{planarbound_even} and \eqref{planarbound_odd}.
\end{proof}

\section{Extremality of the simplex in even dimensions}

\begin{proof}[Proof of Theorem~\ref{thm_simplex}]
Let $\eps$ provide a  sum of maximal norm:
\begin{equation}\label{ueq}
  u = \sum_{i=1}^{d+1} \eps_i u_i.
\end{equation}
Since $d$ is even, and multiplying all coefficients by $-1$ does not change the norm of the sum, we may assume that
\begin{equation}\label{esum}
\sum_{i=1}^{d+1} \eps_i \geq 1.
\end{equation}
By Proposition~\ref{bang}, $\la \eps_i u_i, u \ra \geq 1$
for every $i$. Therefore, since $\eps_i = \pm 1$,
\begin{equation}\label{eiui}
\la (\eps_i + 1) u_i, u \ra \geq   1 + \eps_i
\end{equation}
holds for every $i$ (note that \eqref{eiui} holds trivially for $\eps_i = -1$). Also, using that $\sum u_i = 0$,  \eqref{ueq} leads to
\[
u = \sum_{i=1}^{d+1} (\eps_i+1) u_i.
\]
Therefore, summing \eqref{eiui} over all the indices $i$, and applying \eqref{esum},
\[
|u|^2 = \sum_{i=1}^{d+1} \la (\eps_i +1) u_i, u \ra \geq d+1 + \sum_{i=1}^{d+1} \eps_i \geq d+2,
\]
which is the desired estimate.
\end{proof}

\section{Acknowledgements}

The authors are grateful to A. Polyanski for the inspiring conversations and to D. Hardin and E. Saff for the useful advices.


\begin{thebibliography}{99?}

\bibitem{AN19} G. Ambrus, S. Nietert, {\em Polarization, sign sequences and isotropic vector systems.} Pacific J. Math. {\bf 303} (2019), no. 2, 385--399.

\bibitem{ABG16} G. Ambrus, I. Bárány and V. Grinberg, {\em Small subset sums}. Linear Algebra Appl. {\bf 499} (2016), 66--78.

\bibitem{Ba92} K. M. Ball, {\em Ellipsoids of maximal volume in convex bodies}. Geom. Dedicata, {\bf 41} (1992), no. 2, 241--250.


\bibitem{B97} K. M. Ball, {\em An elementary introduction to modern convex geometry.} In: Flavors of geometry (ed. S. Levy), MSRI publications {\bf 31} (1997), 1--58.


\bibitem{B01} K. M. Ball, {\em Convex Geometry and Functional Analysis}. In: Handbook of the
geometry of Banach spaces, vol.~1 (ed. W. B. Johnson and J. Lindenstrauss),
Elsevier (2001), 161--194.

\bibitem{bdks18} I. Balla, F. Dr\"axler, P. Keevash, B. Sudakov, {\em Equiangular lines and spherical codes in Euclidean space}. Invent. math. {\bf  211} (2018), no. 1, 179--212.

\bibitem{B98} W. Banaszczyk, {\em Balancing vectors and Gaussian measures of $n$-dimensional convex bodies.} Random Structures and Alg. {\bf  12} (1998), 315--360.

\bibitem{B51} Th. Bang, {\em A solution of the ``Plank problem''}. Proc. Amer. Math. Soc. {\bf 2} (1951), 990--993.

\bibitem{B08}I. B\'ar\'any, {\em On the power of linear dependencies.} Building Bridges.
Bolyai Society Mathematical Studies {\bf 19} (2008), 31--45.


\bibitem{BG81}  I. B\'ar\'any and V. S. Grinberg, {\em On some combinatorial questions in finite-dimensional spaces.} Linear Algebra Appl. {\bf 41} (1981), 1--9.

\bibitem{B60}C. Berge, {\em  Probl\`{e}mes de coloration en Théorie des Graphes.} Publ. Inst. Stat. Université de Paris {\bf 9} (1960), 123--160.

\bibitem{BGMPV20+} D. Bilyk, A. Glazyrin, R. Matzke, J. Park, and O. Vlasiuk, { \em Energy on spheres and discreteness of minimizing measures.} arXiv preprint (2019), \texttt{arxiv.org/1908.10354}

\bibitem{BC18} A. Blokhuis, H. Chen, {\em Selectively balancing unit vectors.} Combinatorica {\bf 38} (2018), no. 1, 67 -- 74.

\bibitem{BHS19} S. V. Borodachov, D. P. Hardin, E. B. Saff, {\em Discrete energy on rectifiable sets}. Springer Monographs in Mathematics, Springer,  2019.


\bibitem{B04} K. Böröczky, Jr., {\em Finite packing and covering.} Cambridge Tracts in Mathematics {\bf 154}, Cambridge University Press, 2004.

\bibitem{BG17} R. Brandenberg, B. Gonz\'alez Merino, {\em Minkowski concentricity and complete simplices.}
J. Math. Anal. Appl., 454 (2017), no. 2, 981--994.

\bibitem{BG2} R. Brandenberg, B. Gonz\'alez Merino, {\em The asymmetry of complete and constant width bodies in general normed spaces and the Jung constant.} Israel J. Math., 218 (2017), no. 1, 489--510.

\bibitem{B41} R. Brooks, {\em On colouring the nodes of a network.} Math. Proc. Cambridge Philos. Soc. {\bf 37} (1941), no. 2.,  194-–197.

\bibitem{BFGK18} M. Brugger, M. Fiedler, B. Gonz\'alez Merino, A. Kirschbaum,
{\em Additive colourful Carath\'eodory type results with an application to radii.} Linear Algebra Appl., {\bf 554} (2018), no. 1, 342--357.

\bibitem{CGGKO20}X. Chen, V. Gonzalez, E. Goodman, S. Kang, and A. Okoudjou, {\em Universal optimal configurations for the p-frame potentials.} Adv Comput Math {\bf 46} (2020), no. 4., 1--22.

\bibitem{BC20+} B. Bukh and C. Cox, {\em  Nearly orthogonal vectors and small antipodal spherical codes.} arXiv preprint (2018), \texttt{arxiv.org/1803.02949}

\bibitem{D63} A. Dvoretzky, {\em Problem}. In: Proceedings of Symposia in Pure Mathematics, Vol. 7. Convexity, Amer. Math. Soc., Providence, RI, (1963), p. 496.

%\bibitem{GP} M. Grigorev, A. Polyanskii, {\em personal communication}, 2019.

\bibitem{ol18} M. Grigorev, A. Polyanskii, MIPT Combinatorial Olympiad, Problem 7, 2018.

\bibitem{Ju01} H. Jung, {\em \"Uber die kleinste Kugel, die eine r\"aumliche Figur einschlie$\beta$t.},
J. Reine Angew. Math., {\bf 123} (1901), 241–-257.

\bibitem{KL78} G.A. Kabatiansky, V.I. Levenshtein, {\em Bounds for packings on a sphere and in space.} Probl. Inf. Transm. {\bf 14} (1978), 1–-17.

\bibitem{K68} D.E. Knuth, {\em The Art of Computer Programming.} OKS Print. Reading, Mass: Addison-Wesley Pub. Co, 1968.

\bibitem{P19} A. Polyanskii, private communication, 2019.

\bibitem{S77} J. Spencer, {\em Balancing Games.} J. Combin. Theory Ser. B, {\bf 23} (1977), 68--74.

\bibitem{S81} J. Spencer, {\em Balancing unit vectors.} J. Combin. Theory Ser. A, {\bf 30} (1981), 349--350.

\bibitem{S87}  J. Spencer, {\em Ten Lectures on the Probabilistic Method}. SIAM publications, Philadelphia, 1987.

\bibitem{S13} E. Steinitz, {\em Bedingt konvergente Reihen und konvexe Systeme.} J. Reine Ang. Mathematik, {\bf 143}(1913), 128--175., ibid, {\bf 144}(1914), 1--40., ibid, {\bf 146}(1916), 1--52.

\bibitem{S00} K. Swanepoel, {\em Balancing unit vectors.} J. Combin. Theory Ser. A, {\bf 89} (2000), 105--112.

\bibitem{S16} K. Swanepoel, {\em Sets of unit vectors with small subset sums.} Trans. Amer. Math. Soc. {\bf 368} (2016), 7153--7188.

\bibitem{T32} A. Tarski, {\em Uwagi o stopniu r\'ownowazno\'{s}ci wielokat\'ow}
(Further remarks about the degree of equivalence of polygons), Odbilka Z.
Parametru {\bf  2 } (1932), 310--314.

\bibitem{W74} L. Welch, {\em Lower bounds on the maximum cross correlation of signals.} IEEE Trans. Inform. Theory {\bf 20} (1974), 397--399.

\end{thebibliography}
\end{document}